\documentclass{amsart}
\usepackage{amsfonts, amscd,color,amsmath,amssymb,amsthm}
\usepackage{color, graphicx}
\usepackage{tikz}

\DeclareMathOperator{\Lim}{Lim}

  \DeclareMathOperator{\dom}{dom}

\newtheorem{theorem}{Theorem}
\newtheorem{lemma}[theorem]{Lemma}

\theoremstyle{definition}

\theoremstyle{remark}

\begin{document}

\title{On the descriptive complexity of homogeneous continua}

\author{Pawe{\l}  Krupski}

\email{pawel.krupski@pwr.edu.pl}
\address{Faculty of Pure and Applied Mathematics, Wroc\l aw University of Science and Technology, Wybrze\.{z}e Wyspia\'nskiego 27, 50-370 Wroc\l aw, Poland.}
\dedicatory{Dedicated to Professor Jerzy Mioduszewski on his 90th birthday}
\date{\today}

\begin{abstract}
It is shown that the family of all  homogeneous continua in the hyperspace of all subcontinua of any finite-dimensional Euclidean cube or the Hilbert cube is an analytic subspace of the hyperspace which contains a topological copy of the linear space $c_0=\{(x_k)\in \mathbb R^\omega: \lim x_k=0\}$ as a closed subset.
\end{abstract}

\maketitle
A continuum $X$ is homogeneous if
\begin{multline}\label{h}
\text{for every two points  $x,y\in X$  there is a homeomorphism}\\
\text{$h:X\to X$  such that  $h(x)=y$.}\tag{h}
\end{multline}

Let $I=[0,1]$ and $2\le n\le \omega$. Given a Polish space $(X,d)$, let $2^X$ be the hyperspace of all nonempty compact subsets of $X$ with the Hausdorff metric $H_d$ and  $C(X)\subset 2^X$ be its closed  subspace of all subcontinua in $X$. It is known that $2^X$ is a Polish space and if $X$ is compact then $2^X$ is compact.  We use standard notation $\Lim A_k=A$ if $H_d(A_k,A) \to 0$ in the metric space $(2^X,H_d)$.   Denote by   $\mathcal H(I^n)$ the subspace of $C(I^n)$ consisting of all homogeneous continua in $I^n$.

In 2003, D. P. Bellamy raised the  problem of determining the descriptive complexity of $\mathcal H(I^n)$.  For $n=2$, it is now slightly better understood due to L. C. Hoehn and L. G. Oversteegen~\cite{HO} who settled the long-standing conjecture  that  there are four topological types of  homogeneous plane continua: a point, a circle, the pseudo-arc and the circle of pseudo-arcs. It is also known that  the exact Borel class of the family of all simple closed curves in $I^n$ is  $F_{\sigma\delta}$ in $C(I^n)$~\cite{K1}, the family of all pseudo-arcs in $I^n$ is dense $G_\delta$ in $C(I^n)$~\cite{B} and the family of all circles of pseudo-arcs in $I^n$ is a Borel subset of $C(I^n)$~\cite{CRN} (unfortunately, the actual Borel class  of the latter family  is unknown). So, we conclude  that $\mathcal H(I^2)$ is Borel in $C(I^2)$.

In higher dimensions, there are much more types of homogeneous continua. The Borel complexity of some of them, the most popular ones,  was studied in~\cite{K2,K3} where  some special characterizations of the families were employed. For example, the family of Menger universal curves is exactly $F_{\sigma\delta}$ in $C(I^n)$ for $n\ge 3$,  the family of all solenoids  is a Borel subset of $C(I^n)$ for $n\ge 3$ but we do not know its exact Borel class, similarly for the family of solenoids of  pseudo-arcs or Menger curves of  pseudo-arcs.

These particular results  may suggest that $\mathcal H(I^n)$ is a Borel subset of $C(I^n)$. A naive direct analysis of  formula~\eqref{h} gives the projective class $\Pi^1_2$ for  $\mathcal H(I^n)$.   In this note we reduce the complexity  of $\mathcal H(I^n)$ by showing that it is a $\Sigma^1_1$-subset of $C(I^n)$ which is not $G_{\delta\sigma}$.

\

Recall from~\cite{Ku} that, given compact metric spaces $X,Y$,  the set $P(X,Y)\subset 2^{X\times Y}$ of partial continuous maps $f:Z\to Y$, $Z\in 2^X$ (the maps are identified with their graphs), topologized by the convergence
\begin{multline}\label{partial}
\text{$ f_k \to f$ if and only if $\Lim (\dom f_k)=\dom f$ and}\\
\text{$\lim f_k(x_k)= f(x)$ whenever $x_k\in\dom f_k$, $x\in\dom f$ and $\lim x_k=x$},\tag{*}
\end{multline}
 ($\dom g$ denotes the domain of a partial map $g\in P(X,Y)$), is a Polish space and the subspace $H_p(X,Y)\subset P(X,Y)$ consisting of all partial homeomorphisms is a $G_\delta$ subset of  $P(X,Y)$, so it is also a Polish space.

Let     $Aut_p(I^n)$ denote  the subspace of $H_p(I^n,I^n)$ of all partial autohomeomorphisms $h:X\to X$ where  $X\in C(I^n)$. Given any  $X\in 2^{I^n}$, if $H(X)$ is the space of all autohomeomorphisms $h:X\to X$ with the uniform convergence, then  $H(X)$ is a closed subspace of $Aut_p(I^n)$

\

By a straightforward verification, we get the following

\begin{lemma}\label{l1}
$Aut_p(I^n)$ is a closed subset of $H_p(I^n,I^n)$, so it is a Polish space.
\end{lemma}

Part (1) of the next lemma follows from the Effros' theorem~\cite{Ef};  part (2) is a consequence of part (1) and the fact that open maps on completely metrizable spaces are compact-covering~\cite[Problem 5.5.11]{E}.
\begin{lemma}\label{l2} If $X$ is a homogeneous compactum, then
\begin{enumerate}
\item
for each $a\in X$, the valuation map $\phi_a:  H(X)\to X$, $\phi_a(h)=h(a)$, is an open surjection.
\item
For every closed  $K\subset X$, there exists a compact $C\subset H(X)$ such that $\phi_a(C)= K$.
\end{enumerate}
\end{lemma}

Let $\mathcal F= \{\,(X,a)\in C(I^n)\times I^n: a\in X,\, X\in \mathcal H(I^n)\,\}$.
In view of Lemma~\ref{l2}, the following  equation is obvious but it serves as a key observation  in our evaluation of the complexity of $\mathcal H(I^n)$.
\begin{equation}\label{F}
\mathcal F=\{\,(X,a)\in C(I^n)\times I^n: a\in X,\ \exists\, C\in 2^{Aut_p(I^n)}\ C \subset H(X),\, \phi_a(C)=X\,\}.\tag{$\mathcal F$}
\end{equation}

\begin{lemma}\label{l3}
The set $\{(C,X)\in 2^{Aut_p(I^n)}\times C(I^n): C \subset H(X)\}$ is closed in $2^{Aut_p(I^n)}\times C(I^n)$.
\end{lemma}
\begin{proof}
Let $C_k\subset H(X_k)$,   $C_k\to C$ in $2^{Aut_p(I^n)}$ and $X_k\to X$ in $C(I^n)$. In order to show that $C\subset H(X)$,  take $h\in C$. There exist $h_k\in C_k$ such that $h_k\to h$  in $Aut_p(I^n)$ which means, by~\eqref{partial},  that $\dom h = \Lim (\dom h_k) = \Lim X_k= X$. Since $h$ is assumed to be in $Aut_p(I^n)$, we get  $h\in H(X)$.
 \end{proof}
\begin{lemma}\label{l4}
The set $$\{(C,X,a)\in 2^{Aut_p(I^n)}\times C(I^n)\times I^n: a\in X, \ C\subset H(X), \  \phi_a(C)=X\}$$ is closed in $2^{Aut_p(I^n)}\times C(I^n)\times I^n$.
\end{lemma}
 \begin{proof}
 Suppose $a_k\in X_k$, $C_k\subset H(X_k)$, $\phi_a(C_k)=X_k$ and $a_k\to a$, $X_k\to X$,   $C_k\to C$  in respective spaces. Then $a\in X$ and $C\subset H(X)$ by Lemma~\ref{l3} which implies the inclusion $\phi_a(C)\subset X$.

  In order to  show the converse inclusion, let $x\in X$. There exist a sequence $x_k\in X_k$ converging to $x$ and a sequence $h_k\in C_k$ such that $h_k(a_k)=x_k$. The set $C\cup \bigcup_{k\in \omega}\, C_k$ being compact,  there is a convergent subsequence $(h_{k_i})$. Denote by $h$ its limit in $Aut_p(I^n)$. Then
 $$\dom h=\Lim_i (\dom h_{k_i})= \Lim_i X_{k_i}= X$$    and  $$\lim_i  x_{k_i}=\lim_i  h_{k_i}(a_{k_i})=h(a).$$  It follows that $h\in C$ and $\phi_a(h)=x$.

  \end{proof}

 \begin{lemma}\label{l5}
 The set $\mathcal F$ is $\Sigma^1_1$ in $C(I^n)\times I^n$.
  \end{lemma}
 \begin{proof}
  Refer to equation~\eqref{F} and observe that $\mathcal F$ is the projection of the closed subset $\{(C,X,a): a\in X, \ C\subset H(X), \  \phi_a(C)=X\}$ of the Polish space $2^{Aut_p(I^n)}\times C(I^n)\times I^n$ into the $C(I^n)\times I^n$.
  \end{proof}

  \begin{theorem}
 The space $\mathcal H(I^n)$ is a $\Sigma^1_1$-subset of the hyperspace $C(I^n)$ and $\mathcal H(I^n)$ contains a closed homeomorphic copy of $c_0=\{(x_k)\in \mathbb R^\omega: \lim x_k=0\}$. Consequently, $\mathcal H(I^n)$ is not $G_{\delta\sigma}$.
  \end{theorem}
 \begin{proof}
 $\mathcal H(I^n)$ is the projection of $\mathcal F$ into  $C(I^n)$, therefore it is  $\Sigma^1_1$, by Lemma~\ref{l5}. The second part of the theorem follows directly from
 the proof of an analogous fact in~\cite[Theorem 2 (2)]{K1} (see also~\cite{K1e}). For  reader's convenience, we recall  a  (slightly modified) simple construction of an embedding $f: I^\omega \to C(I^2)$ such that $f(x)$ is a simple closed curve if $x=(x_k)\in \hat{c_0}:=\{(x_k)\in I^\omega: \lim x_k=0\}$ and $f(x)$ is not homogeneous otherwise.
 Let $A_k$ be the segment in $I^2$ from the point $(1/k,0)$ to $(1/(k+1),x_k/2)$ and $B_k$ be the segment from   $(1/(k+1),x_k/2)$ to $(1/(k+1),0)$. Define
 $$f(x)=\bigl(\{0\}\times I\bigr)\cup \bigl( I\times \{1\}\bigr)\cup\bigl(\{1\}\times I\bigr)\cup\bigcup_k(A_k\cup B_k).$$
 Now, $f(\hat{c_0})=f(I^\omega)\cap \mathcal H(I^2)$, hence $f(\hat{c_0})$ is a closed subset of $\mathcal H(I^2)$.   Clearly, we can regard $f$ as an embedding in $C(I^n)$ for $n\ge 2$, so that $f(\hat{c_0})=f(I^\omega)\cap \mathcal H(I^n)$. It is also known that  $\hat{c_0}$ is homeomorphic to $c_0$ because the spaces are both characterized as $F_{\sigma\delta}$-absorbers (see~\cite{DMM}). Thus, $c_0$ embeds in $\mathcal H(I^n)$ as a closed subset . It follows that the space $\mathcal H(I^n)$ is not $G_{\delta\sigma}$ in $C(I^n)$, because otherwise its closed subset $f(\hat{c_0})$ would be $G_{\delta\sigma}$ in $C(I^n)$, so an absolute $G_{\delta\sigma}$-set which is not possible since $c_0$ is a classical example of an absolute $F_{\sigma\delta}$-set which is not absolute $G_{\delta\sigma}$.

 \end{proof}

\subsection*{Acknowledgments}
The author thanks the referee for careful reading  and  corrections.

\bibliographystyle{amsplain}

\end{document}